\pdfoutput=1
\documentclass[11pt]{amsart}
\usepackage{geometry}                
\usepackage{graphicx}
\usepackage{amssymb}
\usepackage{amscd}
\usepackage{epstopdf}
\usepackage{blkarray}
\title{Self-dual almost-K\"ahler four manifolds}

\begin{document}
\author{Inyoung Kim}

\maketitle

 \begin{abstract}
We classify compact self-dual almost-K\"ahler four manifolds of positive type and zero type. 
In particular, using LeBrun's result, we show that any self-dual almost-K\"ahler metric on a manifold which is diffeomorphic to $\mathbb{CP}_{2}$
 is the Fubini-Study metric up to rescaling. 
In case of negative type, we classify compact self-dual almost-K\"ahler four manifolds with $J$-invariant Ricci tensor. 
\end{abstract}
\maketitle
\section{\large\textbf{Introduction}}\label{S:Intro}

On an oriented riemannian four manifold, $\Lambda^{2}$ is decomposed as $\Lambda^{+}\oplus \Lambda^{-}$, 
where $\Lambda^{+}$ is the $(+1)$-eigenspace of Hodge star operator $*$
and $\Lambda^{-}$ is $(-1)$-eigenspace of $*$. 
According to this decomposition, the curvature operator is decomposed as follows. 

\[ 
\mathfrak{R}=
\LARGE
\begin{pmatrix}

 \begin{array}{c|c}
\scriptscriptstyle{W_{+}\hspace{5pt}\scriptstyle{+}\hspace{5pt}\frac{s}{12}}I& \scriptscriptstyle{ric_{0}}\\
 \hline
 
 \hspace{10pt}
 
 \scriptscriptstyle{ric_{0}^{*}}& \scriptscriptstyle{W_{-}\hspace{5pt}\scriptstyle{+}\hspace{5pt}\frac{s}{12}}I\\
 \end{array}
 \end{pmatrix}
 \]

When $W_{-}=0$, we call $g$ is self-dual. If $ric_{0}=0$, then $(M, g)$ is Einstein. 

Let $(M, g)$ be a smooth, riemannian four manifold and let $J$ be an almost complex structure such that 
$g(X, Y)=g(JX, JY)$. Then the the fundamental 2-form is defined by $\omega(X, Y)=g(JX, Y)$. 
Then $(M, g, J)$ is called an almost-Hermitian four-manifold. 
On an almost-Hermitian four-manifold $(M, g, \omega)$, 
$\Lambda^{2}$ is decomposed as follows [5]. 
\[\Lambda^{2}=\omega\oplus[[\Lambda^{2,0}]] \oplus [\Lambda_{0}^{1,1}].\]
Here $[[\Lambda^{2,0}]]$ is the underlying real vector space of type $(2, 0)$ forms 
and $[\Lambda_{0}^{1,1}]$ is the vector space of real $(1, 1)$ forms which is orthogonal to $\omega$.
Moreover, we have $\Lambda^{+}=\omega\oplus[[\Lambda^{2,0}]]$ and $\Lambda^{-}=[\Lambda_{0}^{1,1}]$.

\vspace{40pt}
\hrule

\vspace{20pt}

$\mathbf{Keywords}$: Self-dual, Almost-K\"ahler, Scalar curvature, J-invariant Ricci tensor

$\mathbf{MSC}$: 53C21, 53C25, 57K41

Republic of Korea

Email address: kiysd5@gmail.com

According to this decomposition of $\Lambda^{2}=\omega\oplus[[\Lambda^{2,0}]] \oplus [\Lambda_{0}^{1,1}]$, 
the curvature tensor on an almost Hermitian four manifold is decomposed as follows [5].

\[ 
\mathfrak{R}=
\LARGE
\begin{pmatrix}

\begin{array}{c|c|c}
\scriptscriptstyle{a}& \scriptscriptstyle{ W_{+}^{F}}&\scriptscriptstyle{R_{F}}\\
\hline

\scriptscriptstyle{W_{+}^{F*}}&\scriptscriptstyle{W_{+}^{3}+\frac{1}{2}bI}&\scriptscriptstyle{R_{0}}\\
\hline

\scriptscriptstyle{R_{F}^{*}}&\scriptscriptstyle{R_{0}^{*}}&\scriptscriptstyle{W_{-}+\frac{1}{3}cI}\\

\end{array}

\end{pmatrix}
\]
Here $a=\frac{s^{*}}{4}, b=\frac{s-s^{*}}{4}, c=\frac{s}{4}$, where $s^{*}=2R(\omega, \omega)$. 
If $d\omega=0$ on an almost-Hermitian four manifold $(M, g, J)$, we call $(M, g, \omega)$ is an almost-K\"ahler four manifold. 
In this article, we consider a compact self-dual almost-K\"ahler four manifold $(M, g, \omega)$. 
When $\nabla\omega=0$, equivalently $J$ is integrable, $(M, g, \omega)$ is K\"ahler. 

Compact self-dual K\"ahler surfaces with constant scalar curvature are locally symmetric [8], [10]. 
Moreover, compact self-dual Hermitian surfaces were classified in [3]. 
In this article, we classify compact self-dual almost-K\"ahler four manifolds of zero type and positive type. 
Any conformal class of a compact riemannian four manifold has a conformal metric with constant scalar curvature [6], [27]. 
According to the sign of the scalar curvature, we call a conformal class positive, zero or negative type. 
In particular, we show that a self-dual almost-K\"ahler metric on a manifold which is diffeomorphic to $\mathbb{CP}_{2}$ is 
with the Fubini-Study metric on up to rescaling. 

We also consider compact self-dual almost-K\"ahler four manifolds with $J$-invariant Ricci tensor. 
There exists a general result on weakly self-dual almost-K\"ahler four manifolds with $J$-invariant Ricci tensor [2].
The Ricci form of a weakly self-dual almost-K\"ahler four manifold with $J$-invariant Ricci tensor is a twistor 2-form unless 
$(M, g)$ is Einstein [2]. On the other hand, there does not exist a twistor 2-form on a compact self-dual four manifold with negative type [25]. 
Using this, we get a classification result in case of negative type under the assumption of $J$-invariant Ricci-tensor. 
We show a few applications of this result.

\vspace{50pt}
\section{\large\textbf{Self-dual almost-K\"ahler four manifolds of positive type}}\label{S:Intro}

First, we consider compact self-dual K\"ahler surfaces. 
In [3], compact, self-dual Hermitian surfaces were classified using Miyaoka-Yau inequality [31]. 
Following this, we provide a proof in case of self-dual K\"ahler surfaces. 
We note that Derdzi$\acute{n}$ski's proof does not use this inequality.

\newtheorem{Theorem}{Theorem}
\begin{Theorem}
(Apostolov-Davidov-Muskarov Derdzi$\acute{n}$ski)
Let $(M, g, \omega)$ be a compact self-dual K\"ahler surface. 
Then $(M, g)$ is locally symmetric metric. 
\end{Theorem}
\begin{proof}
We follow the argument given in [9]. 
Note that $\tau\geq 0$. Suppose $\tau=0$. Then $(M, g, \omega)$ is conformally flat K\"ahler.
Then by Tanno's theorem [28], we get the result. 
Suppose $\tau>0$. Note that only minimal surface with $\chi<0$ is $S_{2}$-bundle over $\Sigma_{g}$ for $g\geq 2$, which has $\tau=0$. 
Thus, $\chi\geq 0$. Then $c_{1}^{2}=2\chi+3\tau>0$. 
Then $(M, J)$ is general type or a rational surface. 
Suppose $(M, J)$ is a rational surface. Since $\tau>0$, $(M, J)$ is biholomorphic to $\mathbb{CP}_{2}$. 

Note that $\chi=3\tau$ for $\mathbb{CP}_{2}$ and $\chi\geq3\tau$ for a general type surface [31]. 
Thus, for a self-dual metric, we have 
\[\chi-3\tau=\frac{1}{8\pi^{2}}\int_{M}\left(\frac{s^{2}}{24}-|W_{+}|^{2}-\frac{|ric_{0}|^{2}}{2}\right)d\mu\geq 0.\]
For a K\"ahler surface, we have 
\[\int_{M}|W_{+}|^{2}d\mu=\int_{M}\frac{s^{2}}{24}d\mu.\]
Thus, we get $g$ is Einstein, Thus, $g$ is a self-dual Einstein K\"ahler surface. 
\end{proof}

Below we consider Poon's theorem [25], which was proved by using twistor spaces. 
LeBrun showed Poon's result using Seiberg-Witten theory [17]. 
We prove Poon's Theorem using Gursky's result [13]. 

\begin{Theorem}
(Gursky) Let $(M, g)$ be a compact, smooth oriented riemannian four manifold. 
Suppose $b_{+}\neq 0$ and $(M, g)$ is positive type. 
Then 
\[\int_{M}|W_{+}|^{2}d\mu\geq \frac{4\pi^{2}}{3}(2\chi+3\tau)(M),\]
with equality if and only if $[g]$ contains a K\"ahler-Einstein metric. 
\end{Theorem}

\begin{Theorem}
(Poon) Let $(M, g)$ be a compact simply connected self-dual four manifold with positive scalar curvature. 
If $b_{2}=1$, then $(M, g)$ is conformally equivalent to $\mathbb{CP}_{2}$ with the Fubini-Study metric. 
\end{Theorem}
\begin{proof}
From a Weitzenb\"ock formula for an anti-self-dual 2-form, 
\[\Delta\alpha=\nabla^{*}\nabla\alpha-2W_{-}(\alpha, \cdot)+\frac{s}{3}\alpha,\]
we get $b_{-}=0$. 
By Donaldson's theorem [11], $M$ is homotopy equivalent to $\mathbb{CP}_{2}$.

Let $g$ be a self-dual metric with positive scalar curvature on $M$.
Then by Theorem 2, we have 
\[\int_{M}\left(|W_{+}|^{2}-\frac{s^{2}}{24}+\frac{|ric_{0}|^{2}}{2}\right)d\mu\geq 0.\]
From this, we get 
\[\chi-3\tau=\frac{1}{8\pi^{2}}\int_{M}\left(\frac{s^{2}}{24}-|W_{+}|^{2}+3|W_{-}|^{2}-\frac{|ric_{0}|^{2}}{2}\right)d\mu\leq 0,\]
for a self-dual metric. 
On the other hand, we have $\chi=2+b_{+}=3$ and $\tau=1$. Thus, $\chi-3\tau=0$ on $M$. 
Thus, we get the equality in Theorem 2. 
Therefore, $[g]$ contains a K\"ahler-Einstein metric $(g', J)$. 
Then $(M, J)$ is biholomorphic to $\mathbb{CP}_{2}$ [31]. 
Thus, $g'$ is the Fubini-Study metric on $\mathbb{CP}_{2}$ up to rescaling. 
\end{proof}

Suppose $\mathbb{CP}_{2}$ admits a self-dual metric with zero scalar curvature. Then from 
\[0=\chi-3\tau=\frac{1}{8\pi^{2}}\int_{M}\left(\frac{s^{2}}{24}-|W_{+}|^{2}+3|W_{-}|^{2}-\frac{|ric_{0}|^{2}}{2}\right)d\mu
=\frac{1}{8\pi^{2}}\int_{M}\left(-|W_{+}|^{2}-\frac{|ric_{0}|^{2}}{2}\right)d\mu.\]
Thus, we get $W_{+}=0$. This implies $\tau=0$, which is a contradiction.

As far as we know, it is not known whether $\mathbb{CP}_{2}$ can have a self-dual metric with negative scalar curvature [25]. 
On the other hand, LeBrun showed the same integral inequality in Gursky's theorem 
for an almost-K\"ahler metric 
on the underlying four manifold of a rational or ruled surface with $2\chi+3\tau\geq 0$.
For more general results, we refer to [19]. 

\begin{Theorem}
(LeBrun) Suppose the underlying smooth oriented four manifold $M$ of a rational or ruled surface with $2\chi+3\tau\geq 0$ 
admits an almost-K\"ahler metric $(g, \omega)$. Then, we have 
\[\int_{M}|W_{+}|^{2}d\mu\geq \frac{4\pi^{2}}{3}(2\chi+3\tau)(M),\]
with equality if and only if $[g]$ contains a K\"ahler-Einstein metric with positive scalar curvature. 
In particular, when $2\chi+3\tau=0$, equality does not occur. 
\end{Theorem}

\begin{Theorem}
Let $M$ be the underlying smooth oriented four manifold of a rational or ruled surface with $2\chi+3\tau\geq 0$.
If $M$ admits a self-dual almost-K\"ahler metric $(M, g, \omega)$, then $(M, g, \omega)$
is $\mathbb{CP}_{2}$ with a constant multiple of the Fubini-Study metric.
In particular, a self-dual almost-K\"ahler metric on a manifold $M$ which is diffeomorphic to $\mathbb{CP}_{2}$ 
is the Fubini-Study metric on up to rescaling. 
\end{Theorem}
\begin{proof}
Note that 
\[2\chi+3\tau=\frac{1}{4\pi^{2}}\int_{M}\left(\frac{s^{2}}{24}+2|W_{+}|^{2}-\frac{|ric_{0}|^{2}}{2}\right)d\mu. \]
From Theorem 4, we have 
\[\int_{M}|W_{+}|^{2}d\mu\geq \int_{M}\left( \frac{s^{2}}{24}-\frac{|ric_{0}|^{2}}{2}\right)d\mu.\]
On the other hand, we have 
\[\chi-3\tau=\frac{1}{8\pi^{2}}\int_{M}\left( \frac{s^{2}}{24}-|W_{+}|^{2}+3|W_{-}|^{2}-\frac{|ric_{0}|^{2}}{2}\right)d\mu.\]
For a self-dual metric almost-K\"ahler metric on the underlying smooth four manifold of a rational or ruled surface with $2\chi+3\tau\geq 0$,
we get $\chi\leq 3\tau$. On the other hand, among underlying smooth oriented four manifolds
of rational or ruled surfaces with $2\chi+3\tau\geq 0$, only a manifold diffeomorphic to $\mathbb{CP}_{2}$ has $\chi\leq 3\tau$. 
Note that we have $\chi=3\tau$ on a manifold diffeomorphic to $\mathbb{CP}_{2}$. Using this for a self-dual metric, we have 
\[\int_{M}|W_{+}|^{2}d\mu=12\pi^{2}\tau=\frac{4\pi^{2}}{3}(2\chi+3\tau).\]
Thus, the equality holds in Theorem 4. Therefore, $[g]$ contains a K\"ahler-Einstein metric with positive scalar curvature, 
$(M, g', \omega', J)$. 
Thus, $M$ is biholomorphic to $\mathbb{CP}_{2}$ [31]. 
From $\chi=3\tau$ on $\mathbb{CP}_{2}$, we get a K\"ahler-Einstein metric on  $\mathbb{CP}_{2}$
is self-dual. Thus, $g'$ is the Fubini-Study metric up to rescaling. 
Note that a self-dual harmonic 2-form is a conformally invariant condition 
and therefore, $\omega'$ is a self-dual harmonic 2-form with respect to $g$. 
Since $b_{+}=1$ on $\mathbb{CP}_{2}$, we get $\omega'=c\omega$ for a constant $c$. 
Since $||\omega'||_{g'}=\sqrt{2}$ and $||\omega||_{g}=\sqrt{2}$,
we get 
\[||\omega'||_{g'}=|c|||\omega||_{g'}=||\omega||_{g}=\sqrt{2}.\]
Since $g'$ is conformal to $g$ and $|c|||\omega||_{g'}=||\omega||_{g}$, $g'=|c|g$. 
Thus, $(M, g, \omega)$ is the Fubini-Study metric up to rescaling. .

For $S^{2}$-bundle over $T^{2}$, we have $\chi=\tau=0$. 
Moreover, a self-dual metric $S^{2}$-bundle over $T^{2}$ is anti-self-dual since $\tau=0$. 
Thus, $W_{+}=0$ and therefore the equality occurs with $2\chi+3\tau=0$ in Theorem 4.
On the other hand, by Theorem 4, equality does not occur when $2\chi+3\tau=0$. 
Thus, $S^{2}$-bundle over $T^{2}$ does not admit a self-dual almost-K\"ahler metric. 
\end{proof}

\newtheorem{Corollary}{Corollary}
\begin{Corollary}
Let $(M, g, \omega)$ be a compact self-dual almost-K\"ahler four-manifold with positive type. 
Then $(M, g, \omega)$ is $\mathbb{CP}_{2}$ with the Fubini-Study metric up to rescaling. 
\end{Corollary}
\begin{proof}
Since $(M, g)$ admits a metric of positive scalar curvature, $M$ is diffeomorphic to a rational or ruled surface [23], [24]. 
From the Weitzenb\"ock formula $\Delta\alpha=\nabla^{*}\nabla\alpha-2W_{-}(\alpha, \cdot)+\frac{s}{3}\alpha$, 
we get $b_{-}=0 $ [14]. 
Then we get $M$ is diffeomorphic to $\mathbb{CP}_{2}$ and we get the result from Theorem 5. 
\end{proof}

\vspace{50pt}
\section{\large\textbf{Self-dual almost-K\"ahler four manifolds of zero type}}\label{S:Intro}
Suppose a smooth, compact riemannian manifold $(M, g)$ admits a compatible almost complex structure $J$. 
Then hermitian vector bundles are defined by 
\[\mathbb{V}_{+}=\Lambda^{0,0}\oplus \Lambda^{0,2},\]
\[\mathbb{V}_{+}=\Lambda^{1,1}.\]
Let $\mathbb{S}_{\pm}$ be spin-bundles. 
Then $\mathbb{V}_{\pm}$ can be identified with 
$V_{\pm}=\mathbb{S}_{\pm}\otimes K^{-\frac{1}{2}}$, where $K^{-1}$ is the anti-canonical line bundle. 
Then a unitary connection $A$ on $K^{-1}$ with spin connection on  
$\mathbb{S}_{\pm}$ induces a connection on $\mathbb{V}_{\pm}$. 
Using this connection, the Dirac operator
\[D_{A}:C^{\infty}(\mathbb{V}_{+})\to C^{\infty}(\mathbb{V}_{-})\]
is defined. Then the perturbed Seiberg-Witten equation is

\[D_{A}\Phi=0\]
\[iF_{A}^{+}+\sigma(\Phi)=\eta,\]
where $\rho:\mathbb{V}_{+}\to \Lambda^{+}$ is the canonical real quadratic map such that $|\sigma(\Phi)|^{2}=\frac{|\Phi|^{4}}{8}$ 
and $F_{A}^{+}$ is the self-dual part of the curvature of $F_{A}$ with respect to $g$. 

Let $(g, \eta)$ be a pair of a riemannian metric and a self-dual 2-form with respect to $g$. 
$(g, \eta)$ is called a good pair if $2\pi c_{1}^{+}\neq [\eta_{H}]$. 
Here $c_{1}^{+}$ denotes the projection of $c_{1}$ into the self dual part with respect to $g$
and $\eta_{H}$ is the self-dual harmonic part of $\eta$. 
Because of the equation $iF_{A}^{+}+\sigma(\Phi)=\eta$, for a good pair, 
a solution of the Seiberg-Witten equation is irreducible, namely $\Phi\not\equiv0$.

Then for a good pair $(g, \eta)$ with generic $\eta$, 
we can define the moduli space of solutions of Seiberg-Witten equations and its dimension is given by
 $c_{1}^{2}(K^{-1})-(2\chi+3\tau)=0$. 
We define Seiberg-Witten invariant 
as a mod 2 number of solutions. 
When $b_{+}\geq2$, the Seiberg-Witten invariant is independent of a metric and a perturbation term. 

When $b_{+}=1$, we note that the set of good pairs has two path connected component. 
When $b_{+}=1$, the Seiberg-Witten invariant depends on a component. 
We note that when $c_{1}^{2}>0$, $(g, 0)$ belong to the same component for any riemannian metric $g$. 
If not, then there exists a metric $g'$ such that $c_{1}^{+}$ with respect to $g'$ is zero. 
Then we get $c_{1}^{2}\leq 0$, which is a contradiction.

\newtheorem{Proposition}{Proposition}
\begin{Proposition}
Let $(M, \omega)$ be a compact symplectic four manifold.
Suppose either $c_{1}\cdot\omega<0, c_{1}^{2}>0$, $b_{+}=1$ or $b_{+}\geq 2$ $c_{1}^{+}\neq 0$.
Then for any riemannian metric $g$ with the scalar curvature $s$, we have 
\[\int_{M}s^{2}\geq 32\pi^{2}(c_{1}^{+})^{2}.\]
\end{Proposition}
\begin{proof}
We follow the argument in [15], [16]. 
Let $J$ be an almost-complex structure which is compatible with $\omega$. 
Then an almost-K\"ahler metric by $g'(X, Y):=\omega(X, JY)$. 
Moreover, a homotopy class of $J$ defines a $spin^{c}$-structure. 
Taubes showed [29] that the perturbed Seiberg-Witten equation has a solution for $\eta=r\omega$ with $r$ large. 
\[D_{A}\Phi=0\]
\[iF_{A}^{+}+\sigma(\Phi)=\eta,\]
with respect to an almost-K\"ahler metric $g'$. 
Suppose $c_{1}\cdot\omega<0, c_{1}^{2}>0$, $b_{+}=1$. 
We claim $(g', t\omega)$ is a good pair for $0\leq t\leq r$, and therefore, $(g', 0)$ and $(g', r\omega)$ belongs to the same component. 
Suppose not. 
Then we have $2\pi c_{1}^{+}=t[\omega]$ for some $0\leq t\leq r$. 
Since $c_{1}\cdot\omega<0$ and $t\omega\cdot\omega\geq 0$, we get a contradiction. 
Moreover, since $c_{1}^{2}>0$, $(g, 0)$ and $(g', 0)$ belong to the same component.  
Since $c_{1}^{+}\neq 0$, we get 
for the $spin^{c}$-structure $[J]$, 
the unperturbed Seiberg-Witten equation with respect to $g$ has a non-trivial solution,
\[D_{A}\Phi=0\]
\[iF_{A}^{+}+\sigma(\Phi)=0.\]
Since there exists a non-trivial solution of Seiberg-Witten equation, we have 
\[0=D^{*}_{A}D_{A}\Phi=\nabla_{A}^{*}\nabla_{A}\Phi+\frac{s+|\Phi|^{2}}{4}\Phi.\]
By integrating, we get 
\[0=\int_{M}\left(4|\nabla\Phi|^{2}+s|\Phi|^{2}+|\Phi|^{4}\right)d\mu.\]
From this, we have 
\[\int_{M}s|\Phi|^{2}+|\Phi|^{4}d\mu\leq 0.\]
Using the Schwarz inequality, we get 
\[\int_{M}s^{2}d\mu\int_{M}|\Phi|^{4}d\mu\geq\left(\int_{M}(-s)(|\Phi|^{2})d\mu\right)^{2}\geq [\int_{M}|\Phi|^{4}]^{2}\]
and 
\[\int_{M}s^{2}d\mu\geq \int_{M}|\Phi|^{4}d\mu.\]
From $iF_{A}^{+}=-\sigma(\Phi)$ and $|\sigma(\Phi)|^{2}=\frac{|\Phi|^{4}}{8}$, we get $|iF_{A}^{+}|^{2}=\frac{|\Phi|^{4}}{8}$.
Thus, we have 
\[\int_{M}s^{2}d\mu\geq 8\int_{M}|iF_{A}^{+}|^{2}d\mu.\]
Let $\varphi$ be the harmonic representative of the de Rham class $2\pi c_{1}$.
The harmonic representative of $c_{1}$ minimizes $L^{2}$-norm of closed forms in its de Rham class.
Using 
\[c_{1}^{2}=\frac{1}{4\pi^{2}}\int_{M}\left(|iF_{A}^{+}|^{2}-|iF_{A}^{-}|^{2}\right)d\mu=\frac{1}{4\pi^{2}}\int_{M}\left(|{\varphi}^{+}|^{2}-|\varphi^{-}|^{2}\right)d\mu,\]
we have the following, 
\[\int_{M}|iF_{A}^{+}|^{2}d\mu=\frac{1}{2}\int_{M}(|iF_{A}^{+}|^{2}-|iF_{A}^{-}|^{2})d\mu+\int_{M}(|iF_{A}^{+}|^{2}+|iF_{A}^{-}|^{2})d\mu\]
\[\geq 2\pi^{2}c_{1}^{2}+\frac{1}{2}\int_{M}|\varphi|^{2}d\mu=\int_{M}|\varphi^{+}|^{2}d\mu=4\pi^{2}(c_{1}^{+})^{2}.\]

\end{proof}

\begin{Proposition}
Let $(M, \omega)$ be a compact, symplectic four manifold.
If $M$ admits a metric of zero scalar curvature $g$, then
\begin{itemize}
\item $(M, \omega)$ is rational or ruled; or
\item $(M, g)$ is an anti-self-dual Ricci flat metric. 

\end{itemize}

\end{Proposition}
\begin{proof}
Suppose $M$ admits a metric with nonnegative and not identically zero scalar curvature. 
Then $M$ admits a metric of constant positive scalar curvature [7].
Then $(M, \omega)$ is rational or ruled [23], [24]. 
We note that Taubes showed that $b_{+}=1$ if $c_{1}\cdot\omega>0$ [30]. Moreover, Liu showed that $M$ is diffeomorphic
to a rational or ruled surface if $b_{+}=1$ and $c_{1}\cdot\omega>0$ [23]. 

Suppose $M$ does not admit any metric with nonnegative but not identically zero scalar curvature. 
Then $g$ is Ricci-flat [7]. 
From the following formula, 
\[c_{1}^{2}=2\chi+3\tau=\frac{1}{4\pi^{2}}\int_{M}\left(\frac{s^{2}}{24}+2|W_{+}|^{2}-\frac{|ric_{0}|^{2}}{2}\right)d\mu, \]
we have $c_{1}^{2}\geq 0$. 
If $c_{1}^{2}=0$, then we get $(M, g)$ is anti-self-dual Ricci-flat. 
Suppose $c_{1}^{2}>0$. This in particular implies $c_{1}^{+}\neq 0$. 
Then the Seiberg-Witten invariant is well-defined even when $b_{+}=1$.
Suppose either $c_{1}\cdot\omega<0, b_{+}=1$ or $b_{+}\geq 2$. Then, from Proposition 1, we have
\[\int_{M}s^{2}\geq 32\pi^{2}(c_{1}^{+})^{2}.\]
Since $s=0$, we get $c_{1}^{+}=0$ and therefore, $c_{1}^{2}\leq 0$, which is a contradiction.
Thus, $c_{1}\cdot\omega<0, c_{1}^{2}>0, b_{+}=1$ or $b_{+}\geq 2, c_{1}^{2}>0$ does not occur. 
On the other hand, if $b_{+}=1$ and $c_{1}\cdot\omega=0$, then $c_{1}^{2}\leq 0$. 
Thus, we get $c_{1}^{2}=0$. 
Then we get $(M, g)$ is an anti-self-dual Ricci flat metric. 

We can also prove this result using the following result by LeBrun [18]. 
If $M$ admits a metric with nonnegative but not identically zero scalar curvature, then $(M, \omega)$ is rational or ruled [7], [23], [24]. 
Suppose $M$ does not admit any metric with nonnegative but not identically zero scalar curvature.  
Then in particular, $M$ is not diffeomorphic to a del Pezzo surface
and any scalar flat metric is Ricci-flat [7]. 
Since 
\[c_{1}^{2}=\frac{1}{4\pi^{2}}\int_{M}\left(\frac{s^{2}}{24}+2|W_{+}|^{2}-\frac{|ric_{0}|^{2}}{2}\right)d\mu, \]
we get $c_{1}^{2}\geq 0$. 
Since $M$ is not diffeomorphic to a del Pezzo surface, we get $c_{1}^{2}=0$ by Lemma 1 and therefore, $g$ is anti-self-dual.

\end{proof}

\newtheorem{Lemma}{Lemma}
\begin{Lemma}
(LeBrun)
Let $(M, g, \omega)$ be a compact, symplectic four manifold. 
Suppose $c_{1}^{2}>0$ and $M$ admits a metric of nonnegative scalar curvature. 
Then $M$ is diffeomorphic to a del Pezzo surface. 
\end{Lemma}

\begin{Lemma}
Let $(M, g, \omega)$ be a compact anti-self-dual almost-K\"ahler manifold with zero type. 
Then $(M, g, \omega)$ is anti-self-dual K\"ahler. 

\end{Lemma}
\begin{proof}
On a compact almost-K\"ahler four manifold, we note that the symplectic form is a self-dual harmonic 2-form with constant length $\sqrt{2}$. 
Let $g'$ be a metric with zero scalar curvature which is conformal to $g$. 
Then $g'$ is anti-self-dual with zero scalar curvature. 
Since $g'$ is conformal to $g$ and self-dual harmonic 2-form is conformally invariant, 
$\omega$ is a self-dual harmonic 2-form with respect to $g'$. 
From the Weitzenb\"ock formula, 
\[\Delta\omega=\nabla^{*}\nabla\omega-2W_{+}(\omega, \cdot)+\frac{s}{3}\omega,\]
$\omega$ is parallel with respect to $g'$. 
Then $(M, g', \frac{\omega}{c})$ is K\"ahler, where $||\omega||_{g'}=\sqrt{2}c$ for a positive constant $c$. 
Thus, we get $c||\omega||_{g}=||\omega||_{g'}$. Since $g$ is conformal to $g'$, we get $cg=g'$. 
Then $\omega$ is parallel with respect to $g$. 
Thus, $(M, g, \omega)$ is anti-self-dual K\"ahler. 

\end{proof}

We give one of the proofs of Tanno's Theorem [28] when a manifold is compact.

\begin{Theorem}
(Tanno)
Let $(M, g, \omega)$ be a compact conformally flat K\"ahler surface. 
Then $(M, g)$ is either locally flat or locally a product space of 2-dimensional riemannian manifolds with constant 
curvature $K$ and $-K$.

\end{Theorem}
\begin{proof}
We note that an anti-self-dual K\"ahler surface has zero scalar curvature. 
Since $b_{+}\geq 1$ and $\tau=0$, we have $b_{-}\geq 1$. 
From the Weitzenb\"ock formula for an anti-self-dual 2-form, 
\[\Delta\alpha=\nabla^{*}\nabla\alpha-2W_{-}(\alpha, \cdot)+\frac{s}{3}\alpha,\]
we get $\nabla\alpha=0$. Since $\langle\omega, \alpha\rangle=0$, we get $(M, g)$ is locally a product space
by de Rham decomposition Theorem. Since $s=0$, the Ricci form $\rho(X, Y)=Ric(JX, Y)$ is an anti-self-dual 2-form. 
Since $\rho$ is closed, $\rho$ is a harmonic 2-form. 
Then by the Weitzenb\"ock formula, we get $\nabla\rho=0$ since $s=W_{-}=0$. 
Then, $\nabla Ric=0$. Let $\{e_{1}, e_{2}, e_{3}, e_{4}\}$ be an orthonormal basis such that $\{e_{1}, e_{2}\}$, $\{e_{3}, e_{4}\}$ are basis of each 2-dimensional manifold. 
Since $g$ is locally a product metric, we have $R_{1313}=R_{1414}=0$. 
Since $R_{1212}+R_{1313}+R_{1414}$ is constant, $R_{1212}$ is constant. 
Similarly, $R_{3434}$ is constant. 
Since $s=0$, we get $R_{1212}=-R_{3434}$. 

\end{proof}

\begin{Theorem}
Let $(M, g, \omega)$ be a compact self-dual almost-K\"ahler four manifold with zero type. 
Then
\begin{itemize}
\item $(M, g, \omega)$ is locally flat K\"ahler; or
\item $(M, g, \omega)$ is K\"ahler and $(M, g)$ is locally a product space of 2-dimensional riemannian manifolds of constant curvature $K$ and $-K$. 
\end{itemize}
\end{Theorem}
\begin{proof}
We use Proposition 2. 
Note that from the following formula, 
\[2\chi-3\tau=\int_{M}\left( \frac{s^{2}}{24}+2|W_{-}|^{2}-\frac{|ric_{0}|^{2}}{2}\right)d\mu\leq 0, \]
we have $2\chi-3\tau\leq 0$ for a self-dual metric with zero scalar curvature.
On the other hand, for $\mathbb{CP}_{2}$, we have $2\chi-3\tau=2(2+b_{+})-3b_{+}=3$. 
Thus, a manifold which is diffeomorphic to $\mathbb{CP}_{2}$ does not admit a self-dual metric with zero scalar curvature. 
Any self-dual metric on a rational or ruled surface is anti-self-dual if it is not diffeomorphic to $\mathbb{CP}_{2}$.
Then we can assume that $(M, g, \omega)$ is conformally flat almost-K\"ahler with zero type. 
By Lemma 2, $(M, g, \omega)$ is self-dual and anti-self-dual K\"ahler. Then we get the result by Theorem 6. 

Suppose there is an anti-self-dual Ricci-flat $g'$
which is conformal to $g$. 
Then we get $(M, g, \omega)$ is anti-self-dual almost-K\"ahler with zero type. 
By Lemma 2, $(M, g, \omega)$ is self-dual and anti-self-dual K\"ahler.
On the other hand, since $c_{1}^{2}=0$, $g$ is Ricci-flat. 
Thus, $g$ is locally a flat K\"ahler metric. 
We note that an anti-self-dual Ricci flat metric on $K3$ surface is not self-dual since $\tau\neq 0$ on $K3$ surface. 
\end{proof}

We consider compact self-dual almost-K\"ahler four manifolds with Kodaira dimension zero. 
\begin{Theorem}
(Li [22]) Let $(M, \omega)$ be a compact minimal symplectic four manifold with Kodaira dimension zero. 
Then the Euler number is 0, 12, or 24 and the signature is -16, -8, or 0. 
Moreover, $b_{+}\leq 3$, $b_{-}\leq 19$ and $b_{1}\leq 4$.

\end{Theorem}

\begin{Theorem}
Let $(M, g, \omega)$ be a compact self-dual almost-K\"ahler four manifold with Kodaira dimension $0$. 
Then $(M, g, \omega)$ is locally flat K\"ahler.
\end{Theorem}
\begin{proof}
By definition, for a minimal model of $(M, \omega)$, $c_{1}\cdot\omega=0$ and $c_{1}^{2}=0$ [21]. 
Note that by blowing up, $\chi$ increases by $1$ and $\tau$ decreases by $1$. 
By Theorem 8, the Euler number $\chi$ of $(M, g, \omega)$ is nonnegative. 
Since $c_{1}^{2}=2\chi+3\tau$, by blowing up $c_{1}^{2}$ decreases by $1$. 
Thus, for $(M, \omega)$, $c_{1}^{2}\leq 0$ and $\chi\geq 0$. 
Since $\tau\geq 0$ for a self-dual metric, we get $(M, \omega)$ is minimal and $\tau=0, \chi=0$. 
Thus, $c_{1}\cdot\omega=0$, $c_{1}^{2}=0$ and $(M, g, \omega)$ is self-dual and anti-self-dual almost-K\"ahler. 
For an almost-K\"ahler four manifold, we have $c_{1}\cdot\omega\leq 0$. 
Moreover, $c_{1}\cdot\omega=0$ if and only if $(M, g, \omega)$ is K\"ahler. 
Thus, $(M, g, \omega)$ is self-dual and anti-self-dual K\"ahler. 
Note that we have 
\[\chi=\frac{1}{8\pi^{2}}\int_{M}\left(\frac{{s}^{2}}{24}+|W_{+}|^{2}+|W_{-}|^{2}-\frac{|ric_{0}|^{2}}{2}\right)d\mu.\]
Since $\tau=\chi=s=0$, we get $g$ is Einstein.
\end{proof}

\vspace{50pt}

\section{\large\textbf{Self-dual almost-K\"ahler four manifolds with $J$-invariant Ricci tensor}}\label{S:Intro}
Let $(M, g)$ be a compact, smooth oriented riemannian four manifold and let $\Lambda^{-}M$ be the vector bundle of anti-self-dual 2-forms. 
Then Levi-Civita connection induces the map 
\[\nabla:\Lambda^{-}M\to \Lambda^{-}M\otimes T^{*}M.\]
$T^{*}M\otimes \Lambda^{-}M$ decomposes as $T^{*}M\oplus V$. 
Then the twistor operator $D$ is defined by 
\[D:\Lambda^{-}M\to T^{*}M\otimes \Lambda^{-}M\to V.\]
A twistor 2-form $\Phi$ is an anti-self-dual 2-form such that $\nabla\Phi$ is a section of the bundle $T^{*}M$.
Then from a Weitzenb\"ock formula [25] for a self-dual four manifold 
\[D^{*}D=\nabla^{*}\nabla-as,\]
where $a$ is a positive constant and $s$ is the scalar curvature. 
we get there does not exist a non-trivial twistor 2-form on a compact self-dual four manifold with negative scalar curvature. 
We note that a twistor 2-form is conformally invariant.

Let $(M, g, \omega)$ is a compact almost-K\"ahler four manifold with $J$-invariant Ricci tensor. 
$(M, g, \omega)$ is called weakly self-dual if $\delta W_{-}=0$. 
Let $Ric(X, Y)$ be the Ricci tensor. 
The Ricci form by 
\[\rho(X, Y)=Ric(JX, Y).\]
When the Ricci-tensor is $J$-invariant, $\rho$ is a closed 2-form [12].

\begin{Lemma}
Let $(M, g, \omega)$ be an almost-K\"ahler four manifold with $J$-invariant Ricci tensor. 
Let $\rho(X, Y)=Ric(JX, Y)$ and $\rho_{0}:=\rho-\frac{s}{4}\omega$. 
Then $\rho_{0}$ is an anti-self-dual 2-form. 
Moreover, if $\rho_{0}=0$, then $(M, g)$ is Einstein. 
\end{Lemma}
\begin{proof}
We note that 
\[\rho(JX, JY)=Ric(-X, JY)=Ric(JX, Y)=\rho(X, Y).\]
Since $\rho$ is $J$-invariant and $\rho_{0}\perp\omega$, $\rho_{0}$ is an anti-self-dual 2-form. 
Suppose $\rho_{0}=0$. 
Then we have 
\[0=\rho(X, JY)-\frac{s}{4}\omega(X, JY)=Ric(JX, JY)-\frac{s}{4}\omega(X, JY)=Ric(X, Y)-\frac{s}{4}g(X, Y).\]
Then $(M, g)$ is Einstein.

\end{proof}

Let $h=\frac{1}{2}Ric_{0}+\frac{s}{24}g$, where $Ric_{0}$ is the trace-free part of the Ricci tensor. 
Then the Cotton-York tensor is defined by
\[C_{X, Y}(Z)=-(\nabla_{X}h)(Y, Z)+(\nabla_{Y}h)(X, Z).\]
By the Second Bianchi Identity, we have $\delta W_{-}=C_{-}$, where $C_{-}$ is the anti-self-dual part of $C$. 
If $(M, g, \omega)$ is weakly self-dual, then $C_{-}=0$. 
From $\delta W_{-}=C_{-}=0$ and using $\rho$ is a closed 2-form, 
the following Matsumoto-Tanno identity was shown for a weakly self-dual almost-K\"ahler four manifold with $J$-invariant Ricci tensor [2]. 

\[\nabla_{X}\rho_{0}=-\frac{1}{2}d\tilde{s}(X)\omega+\frac{1}{2}(d\tilde{s}\wedge JX^{\flat}-Jd\tilde{s}\wedge X^{\flat}),\]
where $\tilde{s}=\frac{s}{6}$. 
This identity implies in particular that $\rho_{0}$ is a non-trivial twistor 2-form unless $(M, g)$ is Einstein [2].

\begin{Proposition}
Let $(M, g, \omega)$ be a compact, self-dual almost-K\"ahler four manifold with $J$-invariant Ricci tensor. 
Suppose the scalar curvature is constant and $(M, g)$ is not Einstein. 
Then the scalar curvature is zero. 
Moreover, $(M, g, \omega)$ is K\"ahler and $(M, g)$ is locally a product space of 2-dimensional riemannian manifolds of constant curvature $K$ and $-K$. 
\end{Proposition}

\begin{proof}
Suppose $(M, g, \omega)$ is not Einstein. Then the Ricci form is nontrivial.
From Matsumoto-Tanno identity, we get $\nabla\rho_{0}=0$ if $s$ is constant. 
Let $\overline{M}$ is $M$ with the opposite orientation. 
Then $(\overline{M}, g, \rho_{0})$ is K\"ahler. 
Since $g$ is self-dual on $M$, $(\overline{M}, g)$ is anti-self-dual. 
Thus, $(\overline{M}, g, \rho_{0})$ is an anti-self-dual K\"ahler surface. 
In particular, we get $s=0$. 
Then, we get $(M, g, \omega)$ is a compact, self-dual almost-K\"ahler four-manifold with zero scalar curvature. 
Then we get the result from Theorem 7. 
\end{proof}

\begin{Proposition}
Let $(M, g, \omega)$ be a compact almost-K\"ahler four manifold with $J$-invariant Ricci tensor. 
Suppose $\delta W_{-}=0$ and the scalar curvature is 0.
If $(M, g)$ is not Einstein, 
then $(M, g, \omega)$ is K\"ahler and $(M, g)$ is locally a product space of 2-dimensional riemannian manifolds of constant curvature $K$ and $-K$.

\end{Proposition}
\begin{proof}
From Matsumoto-Tanno identity, we get $\nabla\rho_{0}=0$. 
If $\rho_{0}$ is not trivial, we get a non-trivial parallel 2-form. 
Thus, $(\overline{M}, g, \rho_{0})$ is scalar-flat K\"ahler. 
In particular, $(\overline{M}, g, \rho_{0})$ is anti-self-dual K\"ahler. 
Then $(M, g, \omega)$ is a self-dual almost-K\"ahler four manifold with zero scalar curvature. 
Then the result follows from Theorem 7.

\end{proof}

We consider a compact self-dual almost-K\"ahler four manifold with $J$-invariant Ricci tensor. 
We note that more generally, weakly self-dual almost-K\"ahler four manifolds with $J$-invariant Ricci tensor were classified in [2]. 

\begin{Theorem}
Let $(M, g, \omega)$ be a compact self-dual almost-K\"ahler four manifold with $J$-invariant Ricci tensor. 
Then 
\begin{itemize}
\item $(M, g, \omega)$ is $\mathbb{CP}_{2}$ with the Fubini-Study metric up to rescaling; or 
\item $(M, g, \omega)$ is K\"ahler and $(M, g)$ is locally a product space of 2-dimensional riemannian manifolds of constant curvature $K$ and $-K$; or
\item $(M, g, \omega)$ is locally flat K\"ahler; or
\item $(M, g)$ is an Einstein metric with negative scalar curvature.
\end{itemize}
\end{Theorem}
\begin{proof}

Suppose the type is negative. 
Then there exists a conformally equivalent metric $\tilde{g}$ to $g$ such that 
$(M, \tilde{g})$ has negative scalar curvature. 
Let $D$ be the twistor operator on anti-self-dual two-forms. 
Then from the Weitzenb\"ock formula [25]
\[D^{*}D=\nabla^{*}\nabla-as,\]
where $a$ is positive constant, 
we get there is no twistor 2-form on $(M, \tilde{g})$. 
But a twistor 2-form is conformally invariant, and therefore there is no twistor 2-form on $(M, g)$. 
On the other hand, $\rho_{0}=\rho-\frac{s}{4}\omega$ is a non-trivial twistor two-form by the Matsumoto-Tanno identity 
unless $(M, g, \omega)$ is Einstein [2]. 
Thus the type is nonnegative if $(M, g)$ is not Einstein. 
Suppose the type is positive. 
Then by Corollary 1, we have $(M, g, \omega)$ is the Fubini-Study metric  $\mathbb{CP}_{2}$ up to rescaling.
Suppose the type is zero. Then we get the result from Theorem 7.
\end{proof}

A compact almost-K\"ahler four manifold with $J$-invariant Ricci tensor and nonpositive scalar curvature is either Einstein 
or scalar-flat K\"ahler if Bach tensor vanishes [4]. 
We note that self-duality implies Bach-flatness. 
We prove the following weaker result by assuming self-duality. 

\begin{Corollary}
Let $(M, g, \omega)$ be a compact almost-K\"ahler four manifold with $J$-invariant Ricci tensor. 
Suppose $(M, g)$ is self-dual with nonpositive scalar curvature. 
Then 
\begin{itemize}
\item $(M, g, \omega)$ is K\"ahler and $(M, g)$ is locally a product space of 2-dimensional riemannian manifolds of constant curvature $K$ and $-K$; or
\item $(M, g, \omega)$ is locally flat K\"ahler; or
\item $(M, g)$ is an Einstein metric with negative scalar curvature.
\end{itemize}
\end{Corollary}
\begin{proof}
Suppose $s=0$. 
Then we get the conclusion from Theorem 7. 
Suppose $s$ is negative at least one point. 
Then from the Weitzenb\"ock formula, 
\[D^{*}D=\nabla^{*}\nabla-as,\]
we get there is no non-trivial twistor 2-form. 
Thus, we get $(M, g, \omega)$ is Einstein.

\end{proof}

Below, we get one of the generalizations of Tanno's Theorem in a compact case. 
\begin{Corollary}
Let $(M, g, \omega)$ be a compact conformally flat almost-K\"aher four manifold with $J$-invariant Ricci tensor. 
Then 
\begin{itemize}
\item $(M, g, \omega)$ is K\"ahler and $(M, g)$ is locally a product space of 2-dimensional riemannian manifolds of constant curvature $K$ and $-K$; or
\item $(M, g, \omega)$ is locally flat K\"ahler.

\end{itemize}

\end{Corollary}
\begin{proof}
We use Theorem 10.  Suppose $(M, g, \omega)$ is almost-K\"ahler Einstein with negative scalar curvature. 
For an almost-K\"ahler four manifold, we have 
\[s^{*}:=2R(\omega, \omega)=2\left(W_{+}(\omega, \omega)+\frac{s}{12}|\omega|^{2}\right).\]
From the Weitzenb\"ock formula, we have $0=|\nabla\omega|^{2}-2W_{+}(\omega, \omega)+\frac{s}{3}|\omega|^{2}$. 
This gives 
\[s^{*}=s+|\nabla\omega|^{2}.\]
Thus, for an anti-self-dual almost-K\"ahler four-manifold, we have $s^{*}=\frac{s}{3}$.
On the other hand, there is a point
where $s^{*}$ and $s$ are the same unless $5\chi+6\tau=0$ by Armstrong's theorem [5]. 
Thus, we get either $s\equiv 0$ or $5\chi+6\tau=0$. 
Suppose $5\chi+6\tau=0$. Since $\tau=0$, we get $\chi=0$. 
From the following formula, we note that for a conformally flat Einstein metric, we have $\chi=0$ if and only if $s=0$
\[\chi=\frac{1}{8\pi^{2}}\int_{M}\left(\frac{s^{2}}{24}+|W_{+}|^{2}+|W_{-}|^{2}-\frac{|ric_{0}|^{2}}{2}\right)d\mu.\]
Since $s<0$, we get a contradiction. 
\end{proof}

Let $(M, g, J)$ be an almost-Hermitian manifold. 
A two-plane $X\wedge Y$ is called totally-real if $JX\wedge JY$ is orthogonal to $X\wedge Y$. 
The totally-real sectional curvature is said to be constant at the point $p\in M$ if this curvature is constant for all totally-real two-planes at $p$. 
If the totally-real sectional curvature is constant at each point $p\in M$, $(M, g, J)$ is said to have pointwise constant totally-real sectional curvature.

Let $(M, g, \omega)$ be a compact almost-K\"ahler four-manifold with pointwise constant totally-real sectional curvature. 
Then $(M, g, \omega)$ is self-dual and has $J$-invariant Ricci tensor and $W_{+}^{3}=0$ [4]. 
Such $(M, g, \omega)$ is K\"ahler [1], [2], [4]. 
We show that Theorem 10 gives a different proof of this result. 

\begin{Theorem}
(Apostolov-Calderbank-Gauduchon) 
Let $(M, g, \omega)$ be a compact almost-K\"ahler four-manifold with pointwise constant totally-real sectional curvature. 
Then $(M, g, \omega)$ is self-dual K\"ahler surface. 
\end{Theorem}
\begin{proof}
By Theorem 10, we can assume $(M, g, \omega)$ is almost-K\"ahler Einstein. 
A compact almost-K\"ahler Einstein four manifold 
with $W_{+}^{3}=0$ is K\"ahler [1]. 
From this, we get $(M, g, \omega)$ is a self-dual K\"ahler surface. 
\end{proof}

The first canonical Hermitian connection is defined by 
\[\tilde{\nabla}|_{X}Y:=\nabla_{X}Y-\frac{1}{2}J(\nabla_{X}J)Y.\]
Then the Hermitian holomorphic sectional curvature is defined by 
\[H(Z)=-R^{\tilde{\nabla}}(Z\wedge \bar{Z}, Z\wedge \bar{Z}),\]
for a unit vector $Z\in TM^{1, 0}$. 
If we use Levi-Civita connection, we get the holomorphic sectional curvature. 
$H(Z)$ is constant at the point $p\in M$ if $H(Z)$ is constant $k(p)$ for all $Z\in T^{1, 0}_{p}M$. 
If $H$ is constant at each point $p\in M$, $H$ is said to be pointwise constant. 
Using Theorem 10, we prove the following result by Lejmi and Upmeier [20]. 

\begin{Theorem}
Let $(M, g, \omega)$ be a compact almost-K\"ahler four manifold with 
nonnegative pointwise constant Hermitian holomorphic sectional curvature $k$. 
Then $(M, g, \omega)$ is either to  $\mathbb{CP}_{2}$ with the Fubini-Study metric up to rescaling or locally flat K\"ahler. 
\end{Theorem}
\begin{proof}
The curvature of an almost-Hermitian four manifold decomposes according to the decomposition 
of $\Lambda^{+}=\mathbb{C}\omega\oplus \Lambda^{2. 0}\oplus \Lambda^{0, 2}$, $\Lambda^{-}=\Lambda_{0}^{1, 1}$.

\[ 
\mathfrak{R}=
\LARGE
\begin{pmatrix}

\begin{array}{c|c|c}
\scriptscriptstyle{a}& \scriptscriptstyle{ W_{+}^{F}}&\scriptscriptstyle{R_{F}}\\
\hline

\scriptscriptstyle{W_{+}^{F*}}&\scriptscriptstyle{W_{+}^{3}+\frac{1}{2}bI}&\scriptscriptstyle{R_{0}}\\
\hline

\scriptscriptstyle{R_{F}^{*}}&\scriptscriptstyle{R_{0}^{*}}&\scriptscriptstyle{W_{-}+\frac{1}{3}cI}\\

\end{array}

\end{pmatrix}
\]

The pointwise constant Hermitian holomorphic sectional curvature implies [20]
$(M, g, \omega)$ is self-dual and 
\[|R_{F}|^{2}=\frac{1}{4}\left(\frac{s^{*}-s}{8}\right)^{2}.\]

Moreover, the Hermitian holomorphic sectional curvature is given by [20]
 \[k=\frac{s_{*}}{16}+\frac{5s}{48}.\]

Then we have [20] 
\[c_{1}\cdot\omega=\frac{1}{2\pi}\int_{M}\left(\frac{s^{*}+s}{4}\right)d\mu=\frac{1}{2\pi}\int_{M}\left(\frac{3s^{*}}{16}+\frac{5s}{16}\right)d\mu
+\frac{1}{2\pi}\int_{M}\left(\frac{s^{*}-s}{16}\right)d\mu\]
\[=\frac{1}{2\pi}\int_{M}(3k)d\mu+\frac{1}{2\pi}\int_{M}\left(\frac{s^{*}-s}{16}\right)d\mu.\]
Suppose $s^{*}-s$ is not identically zero. Since $s^{*}-s\geq 0$, we have
$\int_{M}(s^{*}-s)d\mu>0$. Then we get $c_{1}\cdot\omega>0$ if $k\geq 0$. 
Then we get $(M, g, \omega)$ is a rational or ruled surface [23], [30]. 
Since $(M, g)$ is self-dual, we get $M$ is diffeomorphic to $\mathbb{CP}_{2}$ or 
$M$ has $\tau=0$ and therefore, $(M, g, \omega)$ is conformally flat. 
If $M$ is diffeomorphic to $\mathbb{CP}_{2}$, then $(M, g, \omega)$ is the Fubini-Study metric $\mathbb{CP}_{2}$ 
up to rescaling by Theorem 5. This is a contradiction.

Suppose $M$ has $\tau=0$. Then $(M, g, \omega)$ is a conformally flat almost-K\"ahler four manifold.
For an anti-self-dual almost-K\"ahler four-manifold, we have $s^{*}=\frac{s}{3}$.
Moreover, $s\leq 0$, with equality if and only if $(M, g, \omega)$ is K\"ahler. Thus, we have 
\[k=\frac{s}{6}.\]
Since $k\geq0$, we get $s\equiv 0$. 
Therefore, $(M, g, \omega)$ is K\"ahler, which is a contradiction. 
Thus, we get $s^{*}=s$, and $(M, g, \omega)$ is K\"ahler. 

Suppose $k=0$. Then we have $s\equiv 0$. 
Since $(M, g, \omega)$ is K\"ahler, $R_{0}=0$. 
From $|R_{F}|^{2}=\frac{1}{4}\left(\frac{s^{*}-s}{8}\right)^{2}$, we get $R_{F}=0$. 
Thus, $(M, g, \omega)$ is self-dual Ricci-flat K\"ahler, which implies $(M, g, \omega)$ is locally flat K\"ahler. 

Suppose $k>0$ at least one point on $M$. 
Since $(M, g, \omega)$ is K\"ahler, we have $k=\frac{s}{6}$. 
Thus, $s\geq 0$ and $s$ is positive somewhere. 
Then $4\pi c_{1}\cdot\omega=\int_{M}sd\mu>0$. Therefore, $M$ is diffeomrophic to a rational or ruled surface [23], [30]. 
Suppose $b_{-}\neq 0$. Then there exists an anti-self-dual harmonic 2-form $\alpha$. 
From the Weitzenb\"ock formula, we get 
\[\int_{M}\left(|\nabla\alpha|^{2}+\frac{s}{3}|\alpha|^{2}\right)d\mu=0.\]
Since $s\geq 0$, we get $\nabla\alpha=0$ and $\frac{s}{3}|\alpha|^{2}=0$. 
If $\nabla\alpha=0$, then $|\alpha|$ is positive constant unless $\alpha=0$. 
Since $s>0$ lat least one point, we get a contradiction. 
Thus, $\alpha=0$ and therefore $b_{-}=0$. 
Then we get $M$ is diffeomorphic to $\mathbb{CP}_{2}$.
By Theorem 5, a self-dual almost-K\"ahler metric on a manifold which is diffeomorphic to $\mathbb{CP}_{2}$ is the Fubini-Study metric up to rescaling. 
\end{proof}

By assuming $k<0$, Lejmi and Upmeier proved the following result [20]. 
We show this result without assuming $k<0$.

\begin{Theorem}
Let $(M, g, \omega)$ be a compact almost-K\"ahler four manifold of pointwise constant Hermitian holomorphic sectional curvature $k$. 
Suppose $(M, g, \omega)$ has $J$-invariant Ricci tensor.
Then 
\begin{itemize}
\item $(M, g, \omega)$ is $\mathbb{CP}_{2}$ with the Fubini-Study metric up to rescaling; or
\item $(M, g, \omega)$ is locally flat K\"ahler; or
\item $(M, g, \omega)$ is self-dual K\"ahler-Einstein with negative scalar curvature.  
\end{itemize}
\end{Theorem}
\begin{proof}
We use Theorem 10. 
Suppose $(M, g, \omega)$ is K\"ahler and $(M, g)$ is locally a product space of 2-dimensional riemannian manifolds of constant curvature
$K$ and $-K$ with $K\neq 0$. 
Then $(M, g, \omega)$ is conformally flat K\"ahler. 
Thus, $R_{F}=R_{0}=0$. Then $(M, g, \omega)$ is Ricci-flat, which implies $(M, g, \omega)$ is a locally flat K\"ahler space, 
which is a contradiction. 

Suppose $(M, g, \omega)$ is an almost-K\"ahler Einstein with negative scalar curvature. 
For an almost-K\"ahler Einstien manifold, we have $R_{F}=0$. 
From the following formula, 
\[|R_{F}|^{2}=\frac{1}{4}\left(\frac{s^{*}-s}{8}\right)^{2}.\]
we get $s^{*}\equiv s$, and therefore, $(M, g, \omega)$ is K\"ahler. 
Then $(M, g, \omega)$ is self-dual K\"ahler-Einstein with negative scalar curvature. 
\end{proof}

Let $(M, g, \omega)$ be a compact almost-K\"ahler four manifold with pointwise constant holomorphic sectional curvature $k$.
It is known that $(M, g, \omega)$ is self-dual and $\frac{s+3s^{*}}{24}=k$ [26].

\begin{Theorem}
(Sato [26]) Let $(M, g, \omega)$ be a compact almost-K\"ahler four manifold with pointwise constant holomorphic sectional curvature. 
Suppose $(M, g, \omega)$ has $J$-invariant Ricci tensor. 
Then $(M, g, \omega)$ is Einstein. 
\end{Theorem}
\begin{proof}
Except the second case of Theorem 10, all of them are Einstein. 
Thus, suppose $(M, g, \omega)$ is conformally flat K\"ahler
with pointwise constant holomorphic sectional curvature. 
The the scalar curvature is zero. 
The tangent space of each factor of locally a product space is holomorphic. 
Therefore, the curvature of each factor is zero. which is a contradiction. 
\end{proof}

\begin{Theorem}
(Sato [26]) Let $(M, g, \omega)$ be a compact almost-K\"ahler four manifold with constant holomorphic sectional curvature.
Then $(M, g, \omega)$ is self-dual K\"ahler-Einstein. 
\end{Theorem}
\begin{proof}
By Theorem 14, $(M, g, \omega)$ is Einstein. 
Since $\frac{s+3s^{*}}{24}=k$ is constant and $s$ is constant, $s^{*}$ is constant. 
Then $s^{*}-s$ is constant. 
Suppose $s^{*}\not\equiv s$. 
Then by Armstrong's theorem [5], we have $5\chi+6\tau=0$. 
Since $\tau\geq 0$, we get $\chi\leq 0$. 
On the other hand, we have 

\[\chi=\frac{1}{8\pi^{2}}\int_{M}\left(\frac{s^{2}}{24}+|W_{+}|^{2}+|W_{-}|^{2}-\frac{|ric_{0}|^{2}}{2}\right)d\mu\geq 0.\]
Thus, we get $s=W_{+}=0$. An anti-self-dual almost-K\"ahler four manifold with $s=0$ is K\"ahler. 
Thus, we get a contradiction. 
Thus, $s\equiv s^{*}$ and therefore, $(M, g, \omega)$ is a self-dual K\"ahler-Einstein surface.

\end{proof}

\vspace{20pt}
$\mathbf{Acknowledgments}$: The author would like to thank Prof. Claude LeBrun for helpful comments.

\end{document}